\documentclass[11pt, a4paper]{article}
\usepackage{amsfonts}
\usepackage{mathrsfs}

\usepackage{latexsym}
\usepackage{graphicx}

\long\def\delete#1{}

\newtheorem{theorem}{Theorem}
\newtheorem{lemma}[theorem]{Lemma}
\newtheorem{corollary}[theorem]{Corollary}

\newtheorem{observation}[theorem]{Observation}

\usepackage{amssymb}

\def\g{\gamma}

\title{{\bf  Labelling Algorithms for Paired-domination Problems in Block and Interval Graphs}
\thanks{Supported in part by National Natural
Science Foundation of China (Nos. 60673048 and 10471044 ) and
National Basic Research Program of China (2003CB318003).}}

\author{{\bf  Lei Chen$^1$}
\  \  {\bf Changhong Lu$^{2,1,}$}\footnote{Correspond author.
E-mail:
chlu@math.ecnu.edu.cn}\ \ \ \ {\bf Zhenbing Zeng$^1$}\\
      \\ 
       $^1$ Shanghai Key Laboratory of Trustworthy Computing, \\
          East China Normal University, Shanghai, 200062, China \\ 
            \\
      $^2$ Department of Mathematics, \\
       East China Normal University,  Shanghai, 200062, China
       }
\date{}
\begin{document}
\openup 1.0\jot \maketitle

\begin{quote} {\bf Abstract}~ Let $G=(V,E)$ be a graph without
isolated vertices. A set $S\subseteq V$ is a paired-domination set
if every vertex in $V-S$ is adjacent to a vertex in $S$ and  the
subgraph induced by $S$ contains a perfect matching. The
paired-domination problem is to determine the paired-domination
number, which  is the minimum cardinality of a paired-dominating
set. Motivated by a mistaken algorithm given by Chen, Kang and Ng
[ Paired domination on interval and circular-arc graphs, Disc.
Appl. Math. 155(2007),2077-2086], we present two linear time
algorithms to find a minimum cardinality paired-dominating set in
block and interval graphs. In addition, we  prove that
paired-domination problem is {\em NP}-complete for bipartite
graphs, chordal graphs, even split graphs. \vskip 0.3cm
\textbf{Keywords:}~Algorithm, Block graph, Interval graph, Paired-domination, NP-complete
\end{quote}
\newpage
\section{Introduction}

Domination and its variations in graphs have been extensively
studied, cf. \cite{ht1,ht2}.  A set of vertices $S$ is a {\sl
dominating set} for a graph $G=(V,E)$ if every vertex in $V-S$ is
adjacent to a vertex in $S$. The {\sl domination problem} is to
determine the {\sl domination number} $\g (G)$, which is the
minimum cardinality of a dominating set for $G$.

Let $G=(V,E)$ be a simple graph without isolated vertices. For a
vertex $v\in V$, the {\sl open neighborhood} of $v$ is defined as
$N(v)=\{u\in V~|~uv\in E\}$ and the {\sl closed neighborhood} of
$v$ is defined as $N[v]=N(v)\cup \{v\}$.  The distance between $u$
and $v$, denoted by $d_G(u,v)$, is the minimum length of a path
between $u$ and $v$. For a subset $S$ of $V$, the subgraph of $G$
induced by the vertices in  $S$ is denoted by $G[S]$. A {\sl
matching} in a graph $G$ is a set of pairwise nonadjacent  edges
in $G$. A {\sl perfect matching} $M$ in $G$ is a matching such
that every vertex of $G$ is incident to an edge of $M$. Some other
notations and terminology not introduced  in here can be found in
\cite{West}.

 A set $S\subseteq V$ is a {\sl paired-dominating set} of $G$ if
$S$ is a dominating set of $G$ and the induced subgraph $G[S]$ has a
perfect matching. If $e=uv\in M$, where $M$ is a perfect matching of
$G[S]$, we say that $u$ and $v$ are {\sl paired in $S$}. The {\sl
paired-domination problem} is to determine the {\sl
paired-domination number} $\gamma_p(G)$, which  is the minimum
cardinality of a paired-dominating set for a graph $G$. The
paired-domination problem was introduced by Haynes and Slater
\cite{hs}. If we think of each $s\in S\subseteq V$ as the location
of a guard capable of protecting each vertex in $N[S]$, then
``domination" requires every vertex to be protected, and for
paired-domination, we will require the guards' location to be
selected as adjacent pairs of vertices so that each guard is
assigned one other and they are designated as backups for each
other.

Linear time algorithm for paired domination
problem  is available for  trees  \cite{kang2}; Polynomial time
algorithm for paired domination problem  is available for
circular-arc graphs \cite{kang1}. Other results on this subject
can be found in \cite{he,sy}.  Although a linear time algorithm
for paired-domination problem  on interval graphs was  given in
\cite{kang1},  it is incorrect. 
In this paper, we employ the labelling technique to give efficient
algorithms for finding a minimum paired-dominating set in  block
graphs (which contains trees)  and interval graphs. 
In section $2$, we begin with presenting  a linear time algorithm
for paired-domination problem in block graphs and then prove the
correctness of the algorithm. Our algorithm can deduce a quite
simple algorithm for paired-domination problem in trees.  In
section $3$, we first show that the algorithm in \cite{kang1} for
finding a minimum paired-dominating set in interval graphs is
false. Then we give an intuitive algorithm for the
paired-domination problem of interval graphs. In \cite{hs},
authors proved that the paired domination problem is {\em
NP}-complete for undirected graphs. In section $4$, we  show that
it is still {\em NP}-complete  for bipartite graphs, chordal
graphs, even split graphs.


\section{Algorithm for paired-domination problem in Block graphs}

In a graph $G=(V,E)$ with $|V|=n$ and $|E|=m$, a vertex $x$ is a
{\sl cut-vertex} if deleting $x$ and all edges incident to it
increases the number of connected components. A {\sl block} of $G$
is a maximal connected subgraph of $G$ without a cut-vertex. If
$G$ itself is connected and has no cut-vertex, then G is a block.
The intersection of two blocks contains at most one vertex, and a
vertex is a cut-vertex if and only if it is the intersection of
two or more blocks.  A {\sl block graph} is a connected graph
whose  blocks are complete graphs. If every block is $K_2$, then
it is a tree.

As we know, every block graph not isomorphic to complete graph has
at least two {\sl end blocks}, which are blocks with only one
cut-vertex.  Beginning with an end block and working recursively
inward, we can find a vertex ordering $v_1, v_2, \cdots, v_n$ in
{\em O(n+m)} time such that $v_iv_j\in E$ and $v_iv_k\in E$
implies that $v_jv_k\in E$ for $i<j<k\le n$. Note that if $B$ is
an end block with cut-vertex $x$ of block graph $G$, then the
vertices in $B$ is following continually and $x$ is the last
vertex of $B$ in the vertex ordering $v_1,v_2,\cdots, v_n$.

Let $v_1,v_2,\cdots,v_n$ be the vertex ordering of block graph $G$
such that $v_iv_j\in E$ and $v_iv_k\in E$ implies that $v_jv_k\in
E$ for $i<j<k\le n$.
We define the following notations:\\
\\
$1$. $F(v_i)=v_j$, $j=\max\{k~|~v_iv_k\in E,~i<k\}$. $v_j$ is called
the  {\sl father}  of
$v_i$ and $v_i$ is a {\sl child} of $v_j$. 
For technical reasons, we say complete graph has an end block and
$v_n$ is a cut-vertex.
Obviously, $v_j$ must be a cut-vertex in block graphs. \\
$2$. $C(v_i)=\{v_j~|~F(v_j)=v_i\}$.  \\
$3$. $PD(G)$ is a minimum paired-dominating set of $G$.\\
$4$. For any block graph $G$, we define a rooted tree $T(G)$ about
$G$, whose vertex set is $V(G)$, and $uv$ is an edge of $T(G)$ if
and only if $F(u)=v$. The root of $T(G)$ is $v_n$. Moreover let
$T_{v_i}$ is a subtree of $T(G)$ rooted at $v_i$ and every vertex
in $T_{v_i}$ except $v_i$ is a descendant of $v_i$. For a vertex
$v_i\in V(G)$, $D_G(v_i)$ denotes the vertex set consisting of the
descendants of $v_i$ in $T(G)$ and $D_G[v_i]=D_G(v_i)\cup\{v_i\}$.
That is, $D_G[v_i]=V(T_{v_i})$.


In our algorithm, we will use two labels on each vertex $u$, denoted by $(D(u),L(u))$: \\
$\hspace*{10pt} D(u)= \left\{
  \begin{array}{ccl}
     0  & if  & u~is~not~dominated;\\
     1  & if  & u~is~dominated.
  \end{array}
  \right.
$\\
$\hspace*{10pt} L(u)= \left\{
  \begin{array}{ccl}
     0  & if  & u~is~not~put~into~PD(G);\\
     1  & if  & u~is~put~into~PD(G),~but~it~has~ no ~paired~vertex~in~PD(G);\\
     2  & if  & u~is~put~into~PD(G),~and ~it~has~ a ~paired~vertex~in~PD(G).
  \end{array}
  \right.
$

\vskip 0.2cm

 Now, we give an algorithm to determine a minimum paired-dominating set in block
graphs.

\vspace{3mm}

\noindent{\bf Algorithm MPDB.} Find a minimum paired-dominating set
of a block graph.\\
{\bf Input.} A block graph $G=(V,E)$ with a vertex ordering
$v_1,v_2,\cdots ,v_n$ ($n\ge 2$) such that $v_iv_j\in E$ and
$v_iv_k\in E$ implies that $v_jv_k\in E$ for $i<j<k\le n$. Each
vertex $v_i$ has a label $(D(v_i),L(v_i))=(0,0)$. $F(v_i)=v_j$ with
$j=\max\{k~|~v_kv_i\in E,~k>i\}$,
$C(v_i)=\{v_j~|~F(v_j)=v_i\}$.\\
{\bf Output.} A minimum paired-dominating set $PD$ of $G$.

\noindent{\bf Method.}\\
\hspace*{4mm} For $i=1$ to $n$ do\\
\hspace*{8mm} If ($D(v_i)=0$ and $i\neq n$) then\\
\hspace*{12mm} $L(F(v_i))=1$;\\
\hspace*{12mm} $D(u)=1$ for every vertex $u\in N[F(v_i)]$;\\
\hspace*{8mm} endif\\
\hspace*{8mm} If ($D(v_i)=1$) then\\
\hspace*{12mm} Let $C'(v_i)=\{w~|~w\in C(v_i)~and~L(w)=1\}$;\\
\hspace*{12mm} $L(w)=2$ for every vertex $w\in C'(v_i)$;\\
\hspace*{12mm} Let $C''(v_i)=\{w~|~w\in C'(v_i)~ and~ w\in V(M)$,
$M$ is a maximum matching in $G[C'(v_i)]\}$.\\
\hspace*{12mm} If  ($C'(v_i)-C''(v_i)\neq\emptyset$), then\\ 
\hspace*{16mm} $L(v_i)=2$;\\
\hspace*{16mm} $D(u)=1$ for every vertex $u\in N[v_i]$;\\
\hspace*{16mm} Take a vertex $w\in C'(v_i)-C''(v_i)$, for every vertex $v\in C'(v_i)-C''(v_i)-\{w\}$\\
\hspace*{16mm} $L(v')=2$ for some vertex $v'\in C(v)$ such that $L(v')=0$;\\
\hspace*{12mm} endif\\
\hspace*{8mm} endif\\
\hspace*{4mm}  endfor\\
\hspace*{4mm} If ($D(v_n)=0$ or $L(v_n)=1$) then\\
\hspace*{8mm} $L(v_n)=2$;\\
\hspace*{8mm} $L(w)=2$ for some vertex $w\in C(v_n)$ such that
$L(w)=0$;\\
\hspace*{8mm} $D(v_n)=1$;\\
\hspace*{4mm}  endif\\
\hspace*{4mm} Output $PD=\{v| L(v)=2\}$\\ 
 \hspace*{4mm}  end

\vskip 0.3cm

Next we will verify the validity of the algorithm MPDB. For a block
graph $G$ of order $n\ge 2$. when the algorithm MPDB terminates, any
vertex $u\in V(G)$ has a label $D(u)=1$ and any vertex $v\in PD$ has
a label $L(v)=2$. Hence, the output $PD$ is a paired-dominating set
of $G$. It suffices to prove that  $PD$ is a
minimum paired-dominating set of $G$. 

$S_i$ is  the  set of vertices defined by  $v\in S_i$ if and only if
$v$ has the  label $(1,2)$ when  $v_i$ is the considering vertex in
the loop of the algorithm MPDB  for $i=1,2,\cdots,n$. In particular,
we define $S_{n+1}$ as the set of vertices with label $(1,2)$ after
$v_n$ is considered in the loop of the algorithm MPDB. In order to
prove that the output $PD$ is a minimum paired-dominating set of
$G$, we proceed by induction on $i$ and show that, when $v_i$ ($1\le
i\le n$) is the considering vertex in the loop of the algorithm
MPDB, there is a minimum paired-dominating set $S$ in $G$ such that
$S_i\subseteq S$. Obviously, $S_1=\emptyset$. This is certainly true
for $i=1$. Assume that there is a minimum paired-dominating set $S$
in $G$ such that $S_i\subseteq S$ ($1\le i\le n$). We show that
$S_{i+1}$ holds for a given $i+1$ by the following lemmas.



\begin{lemma}\label{lem2}
Let $B$ be an end block of a block graph $G$. If there is a vertex
$u\in V(B)$ such that $u$ is not a cut-vertex with $D(u)=0$ and
$F(u)=v$ ($v\not= u$), then  $L(v)=1$ and $D(w)=1$ for every
vertex $w\in N[v]$. Moreover $PD(G)=PD(G')$, where
$G'=G-(D_G(v)-\{u\})$.
\end{lemma}
\begin{proof}
$G'$ has at least two vertices, hence $PD(G')$ is also a
paired-dominating set of $G$. Then $|PD(G)|\le |PD(G')|$.

Now we show that $|PD(G')|\le |PD(G)|$. If $v\in PD(G)$ and $v$ is
paired with $x\in B$, then $(PD(G)-\{x\})\cup \{u\}$ is also a
paired-domination of $G'$. If $v\in PD(G)$ and $v$ is paired with
$x\notin B$, it is clear that $PD(G)$ is a paired-domination of
$G'$. Assume that $v\notin PD(G)$, then there are two vertices
$x,y\in B$ such that $x,y\in PD(G)$. Hence, $(PD(G)-\{x,y\})\cup
\{u,v\}$ is also a paired-domination of $G'$.

 $v$ must be in $PD(G')$ implies that $v$ is also in $PD(G)$.
Note that the paired vertex with $v$ may be in $B$ or not. So we
 set  $L(v)=1$ and $D(w)=1$ for every vertex $w\in N[v]$.
\ \ $\Box$
\end{proof}
\vskip 0.2cm
 From Lemma \ref{lem2}, when we consider a vertex
$v_i$ in an end block whose label is $(0,0)$ and 
$i\not=n$, then we will put its father $F(v_i)$ into $PD$. Since its
paired vertex can not be determined at this time, we temporarily
 label $L(F(v_i))=1$. Lemma \ref{lem2} implies that  we can
consider a possible smaller block graph $G'$ since $PD(G)=PD(G')$,
where $G'=G-(D_G(v)-\{u\})$.

When we consider a vertex $v_i$ such that $D(v_i)=1$, we will take
its child set to see if there is a child must be paired with
$v_i$. After that we will also  consider a possible smaller block
graph.

\begin{lemma}\label{lem3}
Let $G$ be a block graph 
 and $v_i$ be a considering vertex with  $D(v_i)=1$ in some step of the loop.
 Set $C'(v_i)=\{v_j|~v_j\in C(v_i)~and~L(v_j)=1\}$,  $C''(v_i)=\{w~|~w\in C'(v_i)~ and~ w\in V(M)$,
 $M$ is a maximum
matching in $G[C'(v_i)]\}$  and $S_i(v_j)=S_i\cap D_G(v_j)$ for
$j\le i$. Then, $L(w)=2$ for every vertex $w\in C'(v_i)$.

\noindent$(1)$
\begin{minipage}[t]{155mm}
\setlength{\baselineskip}{18pt} If $L(v_i)=0$ and
$C'(v_i)-C''(v_i)=\emptyset$,
 then
 $PD(G)=PD(G')\cup
S_{i+1}(v_i)$, where  $G'=G-D_G[v_i]\not= K_1$  or
 $G'=G-D_G(v_i)= K_2$.
\end{minipage}

\vspace{2mm}\noindent $(2)$
\begin{minipage}[t]{155mm}
\setlength{\baselineskip}{18pt} If $L(v_i)=1$ and
$C'(v_i)-C''(v_i)=\emptyset$, then
$PD(G)=PD(G')\cup S_{i+1}(v_i)$,  
where $G'=G-(D_G(v_i)-\{x\})$ and $x$ is a vertex whose father is
$v_i$ and $L(x)=0$.

\end{minipage}

\vspace{2mm}\noindent$(3)$
\begin{minipage}[t]{155mm}
\setlength{\baselineskip}{18pt} If
$C'(v_i)-C''(v_i)\not=\emptyset$, then $L(v_i)=2$ and $D(u)=1$ for
every $u\in N[v_i]$,  and for every vertex $v\in C'(v_i)-C''(v_i)$
except one vertex $w$,
 $L(v')=2$ for some vertex $v'\in C(v)$ such that
$L(v')=0$. Moreover,  $PD(G)=PD(G')\cup
(S_{i+1}(v_i)-S_{i+1}(w)
-\{w\})$, where $G'=G-(D_G(v_i)-D_G[w])$.
\end{minipage}
\end{lemma}

\begin{proof}
 $L(w)=1$ for $w\in C'(v_i)$ implies that $w$ must be put into $PD$, so $L(w)=2$ for
 every vertex $w\in C'(v_i)$ and we will determine their paired vertices in this step.

$(1)$ If $v_i$ is not a cut-vertex in graph $G$, then
$S_{i+1}=S_i$ and we do nothing.  Suppose that $v_i$ is a
cut-vertex. 
Obviously, $S_{i+1}(v_i)=S_i(v_i)\cup C'(v_i)$ and $PD(G')\cup
S_{i+1}(v_i)$ is a paired-dominating set of $G$, where
$G'=G-D_G[v_i]\not= K_2$ or $G'=G-D_G(v_i)= K_2$. Now we will show
that $PD(G')\cup S_{i+1}(v_i)$ is a minimum paired-dominating set of
$G$.

By the hypothesis on induction, there exists a minimum
paired-dominating set $S$ of $G$ with $S_i(v_i)\subseteq S$. Since
the label of every vertex in $C'(v_i)$ is $(1,1)$, it must be in a
minimum paired-dominating set of $G$. When $v_i$ is the considering
vertex,  these vertices in $S_i$ have no influence on the label of
remained vertices in possible smaller graph.
We can assume that $S$ is a minimum paired-dominating set with
$S_i(v_i)\cup C'(v_i)\subseteq S$. We first claim that every vertex
in $C'(v_i)$ is paired with another vertex in $C'(v_i)$. Let
$C_p'=\{x|x\not\in C'(v_i)$ and $x$ is paired with a vertex $y$ in
$S$ such that $y\in C'(v_i)\}$. Since vertices in $S_i$ have been
paired each other,   $C_p'\cap S_i=\emptyset$. If $v_i\not\in S$,
since  vertices in $C'(v_i)$ can be paired each other, $S-C_p'$ is a
smaller paired-dominating set of $G$, a contradiction. Then, $v_i\in
C_p'$ and we assume that $v_i$ is paired with $x\in C'(v_i)$. Note
that $|C_p'|\ge 2$. Let $z\in C_p'$ be paired with $y\in C'(v_i)$
and $xy\in E$.  Obviously, $z\in D_G(v_i)$. If $t\in S$ for every
$t\in N(v_i)$, then  $x$ can be paired with $y$ and $S-\{v_i, z\}$
is  a smaller paired-dominating set of $G$, a contradiction. If
there exists some vertex $t\in N(v_i)$ and $t\notin S$, then $S\cup
\{t \}-\{z\}$ is also a paired-dominating set of $G$. Anyway, we get
that these vertices in $C'(v_i)$ are paired each other in $S$.

If $G'=G-D_G(v_i)= K_2$, it is easy to know that $S-S_{i+1}(v_i)$ is
a paired-dominating set of $G'=K_2$. Hence, we know that $PD(G')\cup
S_{i+1}(v_i)=PD(G)$, where $G'=G-D_G(v_i)= K_2$.

Suppose that $G'=G-D_G[v_i]\not=K_1$. We will show that
$S-S_{i+1}(v_i)$ is a paired-dominating set of $G'$. Hence,
$PD(G')\cup S_{i+1}(v_i)=PD(G)$, where $G'=G-D_G[v_i]\not= K_1$.

If $v_i\not\in S$, then it is obvious that $S-S_{i+1}(v_i)$ is a
paired-dominating set of $G'$. Assume that $v_i\in S$ and its paired
vertex in $S$ is $x$. If the father of $v_i$, denoted by $F(v_i)$,
is  in $S$, and its paired vertex is not $v_i$ (that is
$x\not=F(v_i)$), $S-\{x, v_i\}$ is a smaller paired domination set
of $G$, a contradiction. If $F(v_i)$ is not in $S$,  $S-\{x\}\cup
\{F( v_i)\}$ is also a paired-dominating set of $G$. Without loss of
generality, we assume that $v_i$ is paired with $F(v_i)$ in $S$. If
$N(F(v_i))\subseteq S$, then $S-\{v_i,F(v_i)\}$ is a smaller
paired-dominating set of $G$, a contradiction. Assume that $z\in
N(F(v_i))$ is not in $S$. $S-\{v_i\}\cup\{z\}$ is also a
paired-dominating set of $G$. Hence, we can find a minimum paired
domination set $S$ of $G$ such that $S_{i+1}(v_i)\subseteq S$ and
$v_i\notin S$. Note that $D(v_i)=1$. i.e., $v_i$ has been dominated
by some vertex in $S_{i+1}(v_i)$. Hence, $S-S_{i+1}(v_i)$ is a
paired-dominating set of $G'$, where $G'=G-D_G[v_i]\not=K_1$.

$(2)$ When $v_i$ is the considering vertex in the loop of the
algorithm MPDB, $v_i$ has been labelled by $(1,1)$. There exists a
vertex $x\in C(v_i)$ with label $(1,0)$  dominated by $v_i$. In
this situation, $v_i$ must be put into $PD$ after $F(v_i)$ is
 considered in the algorithm MPDB, and  the paired vertex of $v_i$
  may be a vertex in $C(v_i)$. Similarly, we can prove that  those vertices in $C'(v_i)$ are paired
  each other, hence, the paired vertex of $v_i$ may  be a vertex in
  $C(v_i)$ and its label is $(1,0)$.  Thus, let  $G'=G-(D_G(v_i)-\{x\})$, where $x$ is a vertex whose father is
$v_i$ and $L(x)=0$. Using  the same argument, we get that
$PD(G)=PD(G')\cup S_{i+1}(v_i)$. The detail is left to readers.


$(3)$ Since $G[C''(v_i)]$ has a perfect matching and it is also a
maximum matching in $G[C'(v_i)]$, the set $C'(v_i)-C''(v_i)$ is an
independent set. For every vertex $v\in C'(v_i)-C''(v_i)$ except
one $w$, let $v'\in C(v)$ with $L(v')=0$. (It always can be found
since the label of $v$ is (1,1)) and the set of these vertices is
denoted by $CC$, then $|CC|=|C'(v_i)|-|C''(v_i)|-1$.  After $v_i$
was considered in the loop of the algorithm, every vertex in
$C'(v_i)\cup CC\cup \{v_i\}$ has the label $(1,2)$. Hence,
$S_{i+1}=S_i\cup C'(v_i)\cup CC\cup \{v_i\}$.

Let $G'=G-(D_G(v_i)-D_G[w'])$. It is easy to check that $PD(G')\cup
(S_{i+1}(v_i)
-\{w\})$ is a paired-dominating set of $G$.
Next we will show that $PD(G')\cup (S_{i+1}(v_i)
-\{w\})$
is a minimum paired-dominating set of $G$.   We can similarly find a
minimum paired-dominating set $S$ of $G$ such that $S_i\cup
C'(v_i)\subseteq S$.

Let $C_p'=\{x|x\in S$, $x\not\in C'(v_i)$ and its paired vertex is
in $C'(v_i)\}$. Then $|C_p'|\ge |C'(v_i)|-|C''(v_i)|$. If
$v_i\not\in S$, then the set $S-C_p'\cup CC\cup\{v_i\}$ is a
minimum paired-dominating set of $G$. 
If $v_i\in S$ and $v_i\in C_p'$, then $S-(C_p'-\{v_i\})\cup CC$ is
also a minimum paired-dominating set of $G$. If $v_i\in S$ and
$v_i\not\in C_p'$,  we also get that $S-C_p'\cup CC\cup \{w'\}$ is a
minimum paired-dominating set of $G$, where $w'\in C(w)$ has the
label $(1,0)$. Anyway,  we can assume that these vertices in
$C''(v_i)$ are paired in $S$ each other and $\{v_i\}\cup CC\subseteq
S$. And for every $v\in C'(v_i)-C''(v_i)-\{w\}$, $v$ is paired in
$S$ with $v'\in CC$.

Next we will say that $v_i$ and $w$ can be paired in $S$. Let
$w'$($v'$, respectively) is the paired vertex of $w$($v_i$,
respectively) in $S$. We assert that there is a vertex $v''\in
N_G(v')$ such that $v''\not\in S$, for, otherwise, $S-\{w',v'\}$ is
a smaller paired-dominating set of $G$. Hence, $S-\{w'\}\cup\{v''\}$
is also a minimum paired-dominating set of $G$. Thus we can assume
$w$ is paired with $v_i$ in $S$.

Let $S'=S-(S_{i+1}(v_i)-S_{i+1}(w)
-\{w\})$. Then $S'$ is a
paired-dominating set of $G'$, where $G'=G-(D_G(v_i)-D_G[w])$.
Hence, that is $PD(G')\cup (S_{i+1}(v_i)-S_{i+1}(w)
-\{w\})$ is a
minimum paired-dominating set of $G$.  In this case,
$PD(G)=PD(G')\cup (S_{i+1}(v_i)-S_{i+1}(w)
-\{w\})$, where
$G'=G-(D_G(v_i)-D_G[w'])$.

The discussion above implies that $w, v_i$ are paired in $S$, so
let $L(v_i)=2$ and $D(u)=1$ for every vertex in $N[v_i]$. Since
 every vertex $v\in C'(v_i)-C''(v_i)$ except one vertex $w$ is paired with $v'\in
 C(v)$, which is labelled by $(1,0)$, we
 label $v'$ by $(1,2)$ in this step.
\ \  $\Box$
\end{proof}

Using Lemmas \ref{lem2} and \ref{lem3} recursively, we have that
$PD(G)=PD(G')\cup S_{n+1}$ when $v_n$ has been considered in the
loop of the algorithm, where $G'$ is empty or $G'$ is $K_2$. If
$G'$ is $K_2$, it implies that $D(v_n)=0$ or $L(v_n)=1$ after the
loop of the algorithm. For the former, $v_n$ need to be dominated
and no vertex in $N(v_n)$ is put in $PD$. For the latter, $v_n$
must be put in to $PD$ and it need a paired vertex. Hence,  in the
end of the algorithm  MPDB, we change the labels of $v_n$ and $w$
to $(1,2)$, where $w\in C(v_n)$ with the label $(1,0)$.

\begin{theorem}
Given a vertex ordering, described in the beginning of this section,
of a block graph $G$, the algorithm MPDB can produce a minimum
paired-dominating set of $G$ in $O(m+n)$ time, where $m=|E(G)|$ and
$n=|V(G)|$.
\end{theorem}
\begin{proof}
From discussion above,  the algorithm MPDB can produce a minimum
paired-dominating set of a block graph $G$. Furthermore, every
vertex in $V(G)$ and every edge in $E(G)$ are scanned in a constant
number. Although we must find a maximum matching in $G[C'(v_i)]$, it
can be done in linear time, because $G[C'(v_i)]$ is disjoint union
of some clique. If every clique has even number of vertices, then
$G[C'(v_i)]$ has a perfect matching, otherwise it has not a perfect
matching. \ \ $\Box$
\end{proof}
\vskip 0.2cm
 In \cite{kang2}, a linear time algorithm was given to
determine a minimum paired-dominating set in trees. Since block
graphs contain trees, we can also use the algorithm MPDB to produce
a paired-dominating set in trees. The only difference is that
$G[C'(v_i)]$ can not have a perfect matching in tree. Here, we give
a very simple algorithm for tree which can be deduced from
MPDB at once.
\vskip 0.2cm
\noindent{\bf Algorithm MPDT.} Find a minimum paired-dominating set
of a tree.\\
{\bf Input.} A tree $T=(V,E)$ with a vertex ordering
$v_1,v_2,\cdots, v_n$ such that $d_T(v_i,v_n)< d_T(v_j,v_n)$ implies
that $i<j$.  Each vertex $v_i$ has a label $(D(v_i),L(v_i))=(0,0)$.
$F(v_i)=v_j$ with $v_iv_j\in E$ and $j>i$;
$C(v_i)=\{v_j~|~F(v_j)=v_i\}$.\\ 
{\bf Output.} A minimum paired-dominating set $PD$ of $T$.

\noindent{\bf Method.}\\
\hspace*{4mm} For $i=1$ to $n$ do\\
\hspace*{8mm} If ($D(v_i)=0$ and $i\neq n$) then\\
\hspace*{12mm} $L(F(v_i))=1$;\\
\hspace*{12mm} $D(u)=1$ for every vertex $u\in N[F(v_i)]$;\\
\hspace*{8mm}endif\\
\hspace*{8mm} If ($D(v_i)=1$) then\\
\hspace*{12mm} Let $C'(v_i)=\{w~|~w\in C(v_i)$~and~$L(w)=1\}$;\\
\hspace*{12mm} If $C'(v_i)\neq\emptyset$ then\\
\hspace*{16mm} $L(v_i)=2$;\\
\hspace*{16mm} $D(u)=1$ for every vertex $u\in N[v_i]$;\\
\hspace*{16mm} $L(w)=2$ for every vertex $w\in C'(v_i)$;\\
\hspace*{16mm} Take a vertex $w\in C'(v_i)$, for every vertex $v\in C'(v_i)-\{w\}$\\
\hspace*{16mm} $L(v')=2$ for some vertex $v'\in C(v)$ such that $L(v')=0$;\\
\hspace*{12mm} endif\\
 \hspace*{8mm}endif\\
\hspace*{4mm}  endfor\\
\hspace*{4mm} If  ($D(v_n)=0$ or  $L(v_n)=1$)  then\\
\hspace*{8mm} $L(v_n)=2$;\\
\hspace*{8mm} $L(w)=2$ for some vertex $w\in C(v_n)$ such that
$L(w)=0$;\\
\hspace*{8mm} $D(v_n)=1$;\\
\hspace*{4mm}endif\\
\hspace*{4mm} Output $PD=\{v| L(v)=2\}$\\
\hspace*{4mm}  end

\begin{corollary}
Algorithm MPDT can produce a minimum paired-dominating set of a tree
$T$ in $O(m+n)$, where $m=|E(T)|$ and $n=|V(T)|$.
\end{corollary}

\section{Algorithm for paired-domination problem in Interval graph}
An {\sl interval representation} of a graph is a family of intervals
assigned to the vertices so that vertices are adjacent if and only
if the corresponding intervals intersect. A graph having such a
representation is an {\sl interval graph}. Booth and Lueker
\cite{kg} gave an $O(|V(G)|+|E(G)|)$-time algorithm for recognizing
an interval graph and constructing an interval representation using
$PQ$-tree. In \cite{kang1}, a linear algorithm was given to produce
a minimum paired-dominating set of an interval graph. But this
algorithm is incorrect. Next, we will introduce this algorithm and
given a counterexample.

In \cite{kang1}, it is assumed that the input graph is given by an
interval representation $I$ that is a set of $n$ sorted intervals
labelled by $1,2,\cdots,n$ in increasing order of their right
endpoints. The left endpoint of interval $i$ is denoted by $a_i$
and the right endpoint by $b_i~(1<a_i\le b_i\le 2n)$. Define the
following notations.

$(1)$ For a set $S$ of intervals, the largest left(right) endpoint
of the intervals in $S$ is denoted by $\max a(S)~(\max b(S))$; the
interval in $S$ with the largest right endpoint is denoted by
$last(S)$. Let $\max a(S)=0~(\max b(S)=0)$ if $S$ is empty. For
endpoint $e$, use $IFB(e)$ to denote the set of all intervals whose
right endpoint are less than $e$. For any interval $j$, let $l_j$ be
the interval such that intervals $l_j$ and $j$ have nonempty
intersection and $a(l_j)$ is minimum.

$(2)$ For $j\in \{1,2,\cdots,n\}$, define $V_j=\{i:i\in
\{1,2,\cdots,n\}$ and $a_i\le b_j\}$. Let $PD(j)=\{S:S\subseteq
V_j,~S$ is a paired-dominating set of $G[V_j]\}$. Let
$PD(i,j)=\{S:S\subseteq V_j,~S$ is a paired-dominating set of
$G[V_j],~i,j\in S$ and $i,j$ are paired  in  $S\}$. Let
$MPD(j)=Min(PD(j))$, $MPD(i,j)=Min(PD(i,j))$. The left endpoint sets
$A_i=\{a_j:b_{i-1}<a_j<b_j\}$ for $i\in I$, where $b_0=0$.

Introduce two intervals $n+1$ and $n+2$ with
$a_{n+1}=2n+1,~a_{n+2}=2n+2,~b_{n+1}=2n+3$, and $b_{n+2}=2n+4$. Let
$I_p$ be the set of intervals obtained by augmenting $I$ with the
two intervals $n+1$ and $n+2$.

\vspace{3mm}
\noindent{\bf Algorithm MPD}

\noindent {\bf Input:} A set $I_p$ of sorted intervals.

\noindent {\bf Output:} A minimum cardinality paired-dominating set
of $G$ with
interval representation $I_p$.\\
$1$. Find $\max a(IFB(a_j))$ for all $j\in I_p$.\\
$2$. Find $l_j$ for all $j\in I_p$.\\
$3$. Scan the endpoints of $I_p$ to find the left endpoint sets
$A_i=\{a_j:b_{i-1}<a_j<b_j\}$ for $i\in I$, \hspace*{5mm}where
$b_0=0$.\\
$4$. $MPD(0)=\emptyset$.\\
$5$. for $j=1$ to $n+2$ do\\
$6$. Find the left endpoint set $A_k$ containing $\max a(IFB(\min
(a_j,a_{l_j})))$.\\
$7$. Let $b_k$ be the right endpoint of the interval $k$ associated
with the left endpoint set $A_k$.\\
$8$. $MPD(j)=\{l_j,j\}\cup MPD(k)$.\\
$9$. end for\\
Output $MPD(n+2)$.
\vskip 0.3cm
 It is obvious that $MPD(n+2)-\{n+1,n+2\}$ got from
the algorithm MPD is a paired-dominating set of $G$. 
 The following counterexample implies that $MPD(n+2)-\{n+1,n+2\}$
may be not a minimum paired-dominating set of $G$.
\vskip 0.3cm
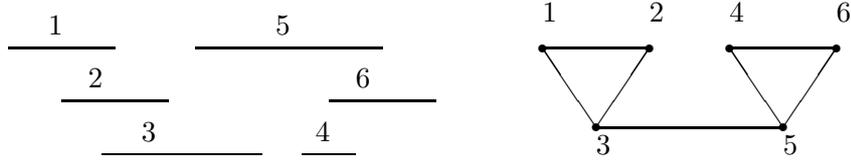
\begin{center}
\begin{picture}(40,40)
 \put(-150,30){\line(1,0){40}}\put(-130,10){\line(1,0){40}}\put(-115,-10){\line(1,0){60}}
 \put(-40,-10){\line(1,0){20}}\put(-80,30){\line(1,0){70}}\put(-30,10){\line(1,0){40}}
 \put(-135,35){$1$}\put(-120,15){$2$}\put(-100,-5){$3$}\put(-35,-5){$4$}\put(-50,35){$5$}\put(-20,15){$6$}

 \put(70,0){\circle*{3}}\put(140,0){\circle*{3}}\put(50,30){\circle*{3}}\put(90,30){\circle*{3}}
 \put(120,30){\circle*{3}}\put(160,30){\circle*{3}}
 \put(70,0){\line(1,0){70}}\put(70,0){\line(-2,3){20}}\put(70,0){\line(2,3){20}}
 \put(140,0){\line(-2,3){20}}\put(140,0){\line(2,3){20}}\put(120,30){\line(1,0){40}}\put(50,30){\line(1,0){40}}
 \put(70,-10){$3$}\put(140,-10){$5$}\put(50,40){$1$}\put(90,40){$2$}\put(120,40){$4$}\put(160,40){$6$}
\end{picture}

\end{center}

\vskip 0.2cm
 The figure above is a counterexample. The left figure
is an interval representation of the graph in the right figure.
The number is ordered by the right endpoint of intervals. The
parameters used in MPD are as follows:
\vskip 0.2cm
\begin{tabular}{|c|c|c|c|c|c|}\hline
$i$ & $a_i$ & $b_i$ & $\max a(IFB(a_i))$ & $l_i$  &  $A_i$       \\\hline
 1  & 0     &  3    &   0                & 2      &  $\{1,2\}$   \\\hline
 2  & 1     &  4    &   0                & 1      &  $\emptyset$ \\\hline
 3  & 2     &  6    &   0                & 1      &  $\{5\}$     \\\hline
 4  & 7     &  9    &   2                & 5      &  $\{7,8\}$   \\\hline
 5  & 5     &  10   &   1                & 3      &  $\emptyset$ \\\hline
 6  & 8     &  11   &   2                & 5      &  $\emptyset$ \\\hline
 7  & 13    &  15   &   8                & 8      &  $\emptyset$ \\\hline
 8  & 14    &  16   &   8                & 7      &  $\emptyset$ \\\hline
\end{tabular}
\vskip 0.2cm
Execute algorithm MPD as follows:\\
$j=1$, $\max a(IFB(\min(a_1,a_2)))=0$, $k=0$, $MPD(1)=\{1,2\}\cup MPD(0)=\{1,2\}$;\\
$j=2$, $\max a(IFB(\min(a_1,a_2)))=0$, $k=0$, $MPD(2)=\{1,2\}\cup MPD(0)=\{1,2\}$;\\
$j=3$, $\max a(IFB(\min(a_3,a_1)))=0$, $k=0$, $MPD(3)=\{1,3\}\cup MPD(0)=\{1,3\}$;\\
$j=4$, $\max a(IFB(\min(a_4,a_5)))=1$, $k=1$, $MPD(4)=\{4,5\}\cup MPD(1)=\{4,5,1,2\}$;\\
$j=5$, $\max a(IFB(\min(a_5,a_3)))=0$, $k=0$, $MPD(5)=\{3,5\}\cup MPD(0)=\{3,5\}$;\\
$j=6$, $\max a(IFB(\min(a_6,a_5)))=1$, $k=1$, $MPD(6)=\{5,6\}\cup MPD(1)=\{5,6,1,2\}$;\\
$j=7$, $\max a(IFB(\min(a_7,a_8)))=8$, $k=4$, $MPD(7)=\{1,2\}\cup MPD(4)=\{7,8,4,5,1,2\}$;\\
$j=8$, $\max a(IFB(\min(a_7,a_8)))=8$, $k=4$, $MPD(8)=\{1,2\}\cup MPD(4)=\{7,8,4,5,1,2\}$;\\

 Hence the result set of the algorithm MPD is
$\{4,5,1,2\}$. But it is easy to see that $\{3,5\}$ is a minimum
paired-dominating set of this graph.  Note that Lemma 2.3 and
Lemma 2.5 in \cite{kang1} are not right. The detail is left to
readers.

 Next, we employ the labelling technique to give  a
linear   algorithms for finding a minimum paired-dominating set of
an interval graph. Let $G=(V,E)$ be an interval graph and its
interval representation is $I$. For every vertex $u_i\in V$, $I_i$
is the corresponding interval, and let $a_i~(b_i$, respectively)
denote the left endpoint (right endpoint, respectively) of interval
$I_i$. We order the vertices of $G$ by $u_1,u_2,\cdots,u_n$ in
increasing order of their left endpoints. Then we have following two
observations.

\begin{observation}\label{ob1}
$u_1,u_2,\cdots,u_n$ is a ordering of an interval graph $G$ by the
increasing order of their left endpoints. If $u_iu_j\in E$ with
$j<i$, then $u_ju_k\in E$ for every $j+1\le k\le i$.
\end{observation}


Let $V_i=\{u_j~|~j\le i\}$ and $G[V_i]$ be an induced subgraph of
$G$. It is obvious that $G[V_n]=G$. Let $F(u_i)=u_j$, where $j=\min
\{k~|~u_ku_i\in E$ and $k<i\}$. In particular, $F(u_1)=u_1$. Let
$w(u_i)=u_j$, where $j=\max \{k~|~u_ku_i\not\in E$ and $k<i\}$. In
particular, If $w(u_i)$ does not exist, we assume that
$w(u_i)=u_0~(u_0\not\in V)$. $PD_i$ denotes a minimum
paired-dominating set of $G[V_i]$. In this paper, we only consider
connected interval graph.

\begin{observation}\label{ob2}
If $G$ is a connected interval graph, then $G[V_i]$ is also
connected.
\end{observation}

\begin{lemma}\label{lem9}
If $F(u_i)\neq u_i$ and $F(F(u_i))=F(u_i)$, then
$PD_i=\{u_i,F(u_i)\}$.
\end{lemma}
\begin{proof}
Since $F(F(u_i))=F(u_i)$, hence $F(u_i)=u_1$. By Observation
\ref{ob1}, for every $1<j\le i$, $u_1u_j\in E$. So
$\{u_i,u_1\}=\{u_i,F(u_i)\}$ is a minimum  paired-dominating set of
$G[V_i]$. Hence, $PD_i=\{u_i,F(u_i)\}$. $\Box$
\end{proof}

\begin{lemma}\label{lem6}
$|PD_{i+1}|\ge |PD_i|$ for $ 2\le i\le n-1$.
\end{lemma}
\begin{proof}
If $u_{i+1}\not\in PD_{i+1}$, then $PD_{i+1}$ is also a
paired-dominating set of $G[V_i]$. So $|PD_{i+1}|\ge |PD_i|$. If
$u_{i+1}\in PD_{i+1}$ and $u_k,u_{i+1}$ are paired in $PD_{i+1}$. We
consider
two cases.\\
{\sl Case 1: $k=i$. }\\
That is $u_i,u_{i+1}$ are paired in $PD_{i+1}$. If
$N_{G[V_i]}(u_i)\subseteq PD_{i+1}$, then $PD_{i+1}-\{u_i,u_{i+1}\}$
is  a paired-dominating set of $G[V_i]$. So
$|PD_{i+1}|>|PD_{i+1}|-2\ge |PD_i|$. If there is a vertex $w\in
N_{G[V_i]}(u_i)$ with $w\not\in PD_{i+1}$, then
$PD_{i+1}-\{u_{i+1}\}\cup \{w\}$ is also a paired-dominating set
of $G[V_i]$. We also get $|PD_{i+1}|\ge |PD_i|$.\\
{\sl Case 2: $k<i$.}\\
If $u_i\not\in PD_{i+1}$, then $PD_{i+1}-\{u_{i+1}\}\cup \{u_i\}$ is
a paired-dominating set of $G[V_i]$, so $|PD_{i+1}|\ge |PD_i|$. If
$u_i\in PD_{i+1}$ and $u_i$ is paired with $u_l$, then $u_lu_k\in E$
and $PD_{i+1}-\{u_i,u_{i+1}\}$ is a paired-dominating set of
$G[V_i]$, so $|PD_{i+1}|>|PD_{i+1}|-2\ge |PD_i|$. $\Box$
\end{proof}

\begin{lemma}\label{lem7}
Let  $F(u_i)=u_k$ and $F(u_k)=u_j$ with $j<k<i$.\\
\begin{tabular}{ll}
$(1)$ & $PD_i=PD_l\cup \{u_j,u_k\}$~ if~ $w(u_j)=u_l$ with $l\ge 2$.\\
$(2)$ & $PD_i=\{u_1,u_2, u_j,u_k\}$~ if~ $w(u_j)=u_1$.\\
 $(3)$ & $PD_i=\{u_j,u_k\}$~ if~ $w(u_j)=u_0$.
\end{tabular}
\end{lemma}
\begin{proof}
$(1)$~ It is obvious that $PD_l\cup \{u_j,u_k\}$ is a
paired-dominating set of $G[V_i]$. It is sufficient to prove
$|PD_i|\ge |PD_l|+2$. Since $F(u_i)=u_k$, there must exist a vertex
$u_{i_1}\in PD_i$ with $k\le i_1\le i$, which dominates $u_i$. We
may assume that $u_{i_1}$ is the last vertex in $PD_i$ which
dominates $u_i$ and  $u_{i_1}$ is paired with $u_{k_1}$. It is
obvious that $k_1\ge j$. Let $l'=\min \{a,~b\}$, where
$w(u_{i_1})=u_a$ and $w(u_{k_1})=u_b\}$. Since $u_lu_{k_1}\not\in E$
and $u_lu_{i_1}\not\in E$~(otherwise $u_lu_j\in E$, a
contradiction), so $l'\ge l\ge 2$. Let $u_c$ is the last vertex in
$PD_i-\{u_{i_1},u_{k_1}\}$. If $c\ge l'$, then
$PD_i-\{u_{i_1},u_{k_1}\}$ is a paired-dominating set of $G[V_c]$.
So $|PD_i|-2\ge |PD_c|$. On the other hand, since $c\ge l'\ge l$, by
Lemma \ref{lem6}, $|PD_c|\ge |PD_l|$. Then $|PD_i|\ge
|PD_l|+2$. 
If $c<l'$, then $PD_i-\{u_{i_1},u_{k_1}\}$ is a paired-domination
set of $G[V_{l'}]$. Since $l'\ge l$, $|PD_i|-2\ge |PD_{l'}|\ge
|PD_l|$. So $|PD_i|\ge |PD_l|+2$. Thus $PD_i=PD_l\cup
\{u_j,u_k\}$.

$(2)$~ Note that $u_1u_2\in E$ and $j\ge 3$ in this situation. So
it is  easy to know that $PD_i=\{u_1,u_2, u_j,u_k\}$~ if~
$w(u_j)=u_1$.

$(3)$~ It is obvious. \ \ \ \ $\Box$
\end{proof}
\vskip 0.2cm
 Now we give  an intuitive algorithm  for determining
a minimum paired-dominating set in interval graphs. \vskip 0.2cm
\noindent{\bf Algorithm MPDI.} Find a minimum paired-dominating set
of an interval graph.\\
{\bf Input.} An interval graph $G=(V,E)$ with a vertex ordering
$u_1,u_2,\cdots,u_n$ ordered by the increasing order of left
endpoints, in which each vertex $u_i$ has a label $D(u_i)=0$. Let
$F(u_i)=u_j~(F(u_1)=u_1)$ such
that $j=\min \{k~|~u_ku_i\in E~and~k<i\}$.\\
{\bf Output.} A minimum paired-dominating set $PD$ of $G$.

\noindent{\bf Method.}\\
\hspace*{4mm} $PD=\emptyset$;\\
\hspace*{4mm} For $i=n$ to $1$ do\\
\hspace*{8mm} If ($D(u_i)=0$) then\\
\hspace*{12mm} If ($F(u_i)\neq u_i$ and
$F(F(u_i))\neq F(u_i)$) then\\
\hspace*{16mm} $PD=PD\cup \{F(u_i),F(F(u_i))\}$;\\
\hspace*{16mm} $D(u)=1$ for every vertex $u\in N[F(u_i)]$;\\
\hspace*{16mm} $D(w)=1$ for every vertex $w\in N[F(F(u_i))]$;\\
\hspace*{12mm} else if ($F(u_i)\neq u_i$) then\\
\hspace*{16mm} $PD=PD\cup \{u_i,F(u_i)\}$;\\
\hspace*{16mm} $D(u)=1$ for every vertex $u\in N[F(u_i)]$;\\
\hspace*{12mm} else \\
\hspace*{16mm} $PD=PD\cup \{u_i,u_2\}$;\\
\hspace*{16mm} $D(u_i)=1$;\\
\hspace*{12mm} endif \\
\hspace*{8mm} endif \\
\hspace*{4mm} endfor

\begin{theorem}
Given a vertex ordering ordered by the increasing order of left
endpoints, the algorithm MPDI can produce a minimum
paired-dominating set of $G$ in $O(m+n)$, where $m=|E(G)|$ and
$n=|V(G)|$.
\end{theorem}
\begin{proof}
By  Lemmas \ref{lem9} and  \ref{lem7}, we know that the algorithm
MPDI can produce a minimum paired-dominating set of an interval
graph $G$. Since each vertex and edge are scanned in a constant
number, hence the algorithm MPDI can finish in $O(m+n)$, where
$m=|E(G)|$ and $n=|V(G)|$.  $\Box$
\end{proof}

\section{NP-completeness of paired-domination problem}
A graph is {\sl chordal} if every cycle of length greater than
three has a chord, i.e. an edge jointing two nonconsecutive
vertices in the cycle. A graph is {\sl split} if its vertex set
can be partitioned into a stable set and a clique.  Obviously,
block graphs, interval graphs and split graphs are three
subclasses of chordal graphs.
 This section establishes {\em NP}-complete results for the
paired-domination problem in bipartite graphs and  chordal graphs.
The transformation is from the vertex cover problem, which is
known to be {\em NP}-complete. The {\sl vertex cover problem} is
for a given nontrivial graph and a positive integer $k$ to answer
if there is a vertex set of size at most $k$ such that each edge
of the graph has at least one end vertex in this set.

\begin{theorem}\label{thm1}
Paired-domination problem is  NP-complete for bipartite graphs.
\end{theorem}
\begin{proof}
 For a
bipartite graph $G=(V,E)$, a positive even integer $k$, and an
arbitrary subset $S\subseteq V$ with $|S|\le k$, it is easy to
verify in polynomial time whether $S$ is a paired-dominating set
of $G$. Hence,  paired-domination problem is in {\em NP}.

We  construct a reduction from the vertex cover problem. Given a
nontrivial graph $G=(V,E)$, where $V=\{v_1,v_2,\cdots, v_n\}$ and
$E=\{e_1,e_2,\cdots,e_m\}$, Let $V_i=\{v_1^i,v_2^i,\cdots,
v_n^i\}$ and $E_i=\{e_1^i,e_2^i,\cdots,e_m^i\}~(i=1,2)$. Construct
the graph $G'=(V',E')$ with vertex set $V'=V_1\cup V_2\cup E_1\cup
E_2$, and edge set $E'=\{uv|~u\in V_1$ and $v\in V_2\}\cup
\{v_j^ie_k^i|~i=1,2$ and $v_j$ is incident to $e_k$ in $G \}$.
Note that $G'$ is a bipartite graph.

Next, we will show that $G$ has a vertex cover of size at most $k$
if and only if $G'$ has a paired-dominating set of size at most
$2k$. Let $VC=\{v_{i_1},v_{i_2},\cdots, v_{i_k}\}$ be a vertex
cover of $G$. Then it is obvious that
$\{v_{i_1}^1,v_{i_2}^1,\cdots,
v_{i_k}^1\}\cup\{v_{i_1}^2,v_{i_2}^2,\cdots, v_{i_k}^2\}$ is a
paired-dominating set of $G'$ and its size is $2k$. For the
converse, let $PD$ be a paired-dominating set of $G'$ with
$|PD|\le  2k$. Obviously, we can assume that $k\le n-1$ since
$V=\{v_1,v_2,\cdots, v_n\}$ is a vertex cover of $G$. If $PD\cap
E_1\not=\emptyset$, without loss of generality, we assume that
$e^1_1\in PD$ and its paired vertex in $PD$ is $v^1_i$. Since
$k\le n-1$, there exists a vertex $v^2_j\notin PD$. Hence,
$PD\cup\{v^2_j\}-\{e^1_1\}$ is also a paired-dominating set of
$G'$.  Then we may assume that $PD\cap(E_1\cup E_2)=\emptyset$.
Suppose that $VC_1=PD\cap V_1$, and note that $|VC_1|\le k$. Let
$VC=\{v_i|~v_i^1\in VC_1\}$ and $VC$ is a vertex cover of $G$ such
that $|VC|\le k$.

Finally, one can construct $G'$ from $G$ in polynomial time. This
implies that paired-domination problem is {\em NP}-complete for
bipartite graphs. $\Box$
\end{proof}


\begin{theorem}\label{thm2}
Paired-domination problem is   NP-complete for chordal graphs.
\end{theorem}
\begin{proof}
We still construct a reduction from the vertex cover problem.
Given a nontrivial graph $G=(V,E)$, where $V=\{v_1,v_2,\cdots,
v_n\}$ and $E=\{e_1,e_2,\cdots,e_m\}$, Let
$V_i=\{v_1^i,v_2^i,\cdots, v_n^i\}~(i=1,2)$ and
$E_i=\{e_1^i,e_2^i,\cdots,e_m^i\}~(i=1,2)$. Construct the graph
$G'=(V',E')$ with vertex set $V'=V_1\cup V_2\cup E_1\cup E_2$, and
edge set $E'=\{uv|~u\in V_1\cup V_2$ and $u\not=v\}\cup
\{v_j^ie_k^i|~i=1,2$ and $v_j$ is incident to $e_k$ in $G \}$.
Note that $G'$ is a chordal graph.

It is straightforward to show that $G$ has a vertex cover of size
at most  $k$ if and only if $G'$ has a paired-dominating set of
size at most $2k$. The proof is almost similar with that of
Theorem \ref{thm1}. In here, we can also assume that $k\le n-1$
and $PD\cap(E_1\cup E_2)=\emptyset$. So, either $VC_1=PD\cap V_1$
or $VC_2=PD\cap V_2$ has
 size at most $k$. The detail is left to readers. \ \ $\Box$
\end{proof}
\vskip 0.3cm
 Note that $G'$ in Theorem \ref{thm2} is  also a split
graph. Hence we get a stronger result as follows.

\begin{corollary}
Paired-domination problem is  NP-complete for split graphs.
\end{corollary}

\small {

}

\end{document}